\documentclass[10pt,reqno]{amsart}
\usepackage{amsmath,amsopn,amssymb,amsthm}
\usepackage[cp1251]{inputenc}
\usepackage[english]{babel}
\usepackage[pagebackref,breaklinks=true,colorlinks=true,linkcolor=blue,citecolor=blue,urlcolor=blue]{hyperref}
\usepackage{graphicx}%
\usepackage{color}

\newtheorem{theorem}{Theorem}[section]

\newtheorem{lemma}[theorem]{Lemma}

\renewenvironment{proof}[1][Proof]{\noindent\textbf{#1.} }{\ \rule{0.5em}{0.5em}}

\begin{document}

\title[Conformal vector fields on compact connected homogeneous Finsler manifolds]{Conformal vector fields on compact connected homogeneous Finsler manifolds}
\author{Ming Xu}
\address[Ming Xu] {School of Mathematical Sciences,
Capital Normal University,
Beijing 100048,
P.R. China}
\email{mgmgmgxu@163.com}

\begin{abstract}
Let $(M,F)$ be a compact connected homogeneous non-Riemannian Finsler manifold
with $\dim M>1$. We prove that any conformal vector field on $(M,F)$ is a Killing vector field.
Further more, we prove that $\rho F$ is a homogeneous Finsler metric on $M$ if and only if $\rho$ is a positive constant function.


\end{abstract}

\maketitle

\section{Introduction}
Let $(M,g)$ be a connected Riemannian manifold.
We denote by $C(M,g)$ and $I(M,g)$ the groups of all conformal transformations and all isometries, and by $C_0(M,g)$ and $I_0(M,g)$ their identity components. Roughly speaking,
they are Lie transformation groups (acting smoothly on $M$), and their Lie algebras are the spaces  of all complete conformal vector fields and all complete Killing vector fields respectively \cite{KN1963}\cite{MS1939}.
Notice $C(M,g)$ may have an infinite dimension when $\dim M=2$. The conformality in Riemannian geometry has been studied for many decades \cite{Al1972}\cite{Fe1996}\cite{GK1962}\cite{Le1971}\cite{Ob1971}\cite{YN1959}. Using warped product or other techniques, it is easy to construct non-Killing conformal vector fields, which
reveals $C_0(M,g)\neq I_0(M,g)$ in general. However, when $(M,g)$ is a compact connected homogeneous Riemannian manifold with $\dim M>1$,
$C_0(M,g)\neq I_0(M,g)$ implies the following rigidity.

\begin{theorem} \label{thm-1}
Let $(M,g)$ be a compact connected homogeneous Riemannian manifold with
$\dim M>1$. Suppose that $C_0(M,g)\neq I_0(M,g)$,
then $(M,g)$ is isometric to a sphere with constant curvature.
\end{theorem}

Theorem \ref{thm-1} is an immediate corollary of Lichnerowicz Theorem when $\dim M\geq3$ \cite{Li1964}, and it follows follows after some simple observations when $\dim M=2$ (see Section \ref{subsection-3-2}). See also \cite{GK1962} where Theorem \ref{thm-1} is proved
when $\dim M>3$.

The purpose of this paper is to generalize Theorem \ref{thm-1} to Finsler geometry.
On a Finsler manifold $(M,F)$, the conformal transformation group $C(M,F)$, the isometry group $I(M,F)$, and their identity components, $C_0(M,F)$ and $I_0(M,F)$ respectively,
can be similarly defined \cite{De2012}\cite{Fe1980}\cite{Kn1929}.
There are various average processes \cite{DX2016}\cite{MRTZ2009}\cite{MT2012} which can produce a Riemannian metric from a Finsler one while preserve or even enlarge
the conformal transformation group and the isometry group. Using this average technique,
we can quickly verify that most $C(M,F)$ are Lie transformation groups (see Lemma \ref{lemma-2} and Lemma \ref{lemma-4}), and it follows naturally that their Lie algebras $\mathfrak{C}(M,F)$ are the spaces of complete conformal vector fields.
The main theorem of this paper is
\begin{theorem}\label{thm-2}
Let $(M,F)$ be a compact connected homogeneous non-Riemannian Finsler manifold with $\dim M>1$. Then $C_0(M,F)=I_0(M,F)$, i.e., any conformal vector field on $(M,F)$ is a Killing vector field.
\end{theorem}
Theorem \ref{thm-2} can also be formulated as
\begin{theorem}\label{cor-1} Let $(M,F)$ be a compact connected non-Riemannian homogeneous Finsler manifold with $\dim M>1$, and $\rho$ a positive smooth function on $M$. Then $\rho F$ is homogeneous if and only if $\rho$ is constant.
\end{theorem}

%
%
The analytic method proving Theorem \ref{thm-1} is not valid in the Finsler context.
The proof of Theorem \ref{thm-2} has three other ingredients, i.e.,  the average process,
the Finsler Confromal Lichnerowicz-Obata Conjecture (proved by V.S. Matveev et al \cite{MRTZ2009}),
and the Lie theory.
In Section 2, we summarize some necessary knowledge in Finsler geometry.
In Section 3, we introduce the average process and use it to discuss the Finslerian conformal transformation group. In Section 4, we prove Theorem \ref{thm-2} and Theorem \ref{cor-1}.

\section{Preliminaries}
\subsection{Minkowski norm and Finsler metric}

A {\it Minkowski norm} $F$ on a finite dimensional real vector space $\mathbf{V}$
is a continuous function $F:\mathbf{V}\rightarrow[0,+\infty)$ satisfying the following requirements:
\begin{enumerate}
\item Regularity: $F$ is positive and smooth on $\mathbf{V}\backslash\{0\}$;
\item Positive 1-homogeneity: $F(\lambda y)=\lambda F(y)$, $\forall \lambda\geq0$, $y\in\mathbf{V}$;
\item Strong convexity: given any basis $\{e_1,\cdots,e_n\}$ of $\mathbf{V}$ and the corresponding linear coordinate
$y=y^ie_i\Leftrightarrow(y^1,\cdots,y^n)$, the Hessian matrix $(g_{ij}(y))=([\tfrac12F^2]_{y^iy^j})$
is positive definite when $y\neq0$, or equivalently,  the {\it fundamental tensor}
$$g_y(u,v)=\tfrac12\tfrac{\partial^2}{\partial s\partial t}|_{s=t=0}F^2(y+su+tv),\quad \forall  u,v\in \mathbf{V},$$
is an inner product on $\mathbf{V}$ for any $y\in\mathbf{V}\backslash\{0\}$.
\end{enumerate}
We call a Minkowski norm $F$ {\it Euclidean}
when $g_y(\cdot,\cdot)$ in (3) is independent of $y$.

A Finsler metric $F$ on a smooth manifold $M$ is a continuous function $F: TM\rightarrow[0,+\infty)$ such that $F|_{TM\backslash0}$ is smooth and $F|_{T_xM}$ is a Minkowski norm for each $x\in M$ \cite{BCS2000}. We usually use the pair $(M,F)$ to denote a {\it Finsler manifold}. We call a Finsler metric $F$ {\it Riemannian} if it is Euclidean in
each tangent space. In this case, we also use $g=F^2$ to denote the metric.

\subsection{Finslerian conformality and isometry}

It is obvious to see that, for any Finsler metric $F$ and positive smooth function $\rho$ on $M$,
$\rho F$ is also a Finsler metric on $M$. Here we call $F$ and $\rho F$ {\it conformally equivalent}.
A diffeomorphism $f$ on a Finsler manifold $(M,F)$ is called a {\it conformal transformation} if $F$ and $f^* F$ are conformally equivalent. In particular, we call $f$ {\it an isometry} if $f^*F=F$.

Conformal transformations and isometries can be infinitesimal generated by conformal and Killing vector fields respectively. A smooth tangent vector field $V$ on $(M,F)$ is called {\it conformal} if
there exists a smooth function $\rho$ on $M$, such that $\mathcal{L}_VF=\rho F$ (here $\mathcal{L}$ is the Lie derivative).  In particular, a conformal vector field $V$ is a {\it Killing vector field} when $\mathcal{L}_VF=0$.
A conformal or Killing vector field can generate a globally defined one-parameter subgroup of conformal
transformations or isometries only when it is complete. Since we only discuss compact manifolds in this paper, the completeness is automatically satisfied.

We denote by $C(M,F)$, $I(M,F)$, $C_0(M,F)$, $I_0(M,F)$, $\mathfrak{C}(M,F)$ and $\mathfrak{I}(M,F)$ the group of all conformal transformations, the group of all isometries,
the identity component of $C(M,F)$, the identity component of $I(M,F)$, the space of all complete conformal vector fields, the space of all complete Killing vector fields on $(M,F)$, respectively.
\section{Conformality for a homogeneous Finsler manifold}
\subsection{Average process and group of conformal transformations}


Let $F$ be a Minkowski norm on $\mathbf{V}$. Averaging its fundamental tensor along the indicatrix
$S=\{y|y\in\mathbf{V},F(y)=1\}$, we can get an Euclidean norm $\overline{F}$.
Here we only recall the
construction in Section 2.3 of \cite{DX2016}.
See \cite{MRTZ2009}\cite{MT2012} for more constructions which can also meet our needs.

Let $g$ be the Riemannian metric on $\mathbf{V}\backslash\{0\}$ determined by the fundamental tensor of $F$ (we call it the {\it Hessian metric} of $F$), and denote by $d\mu$ the volume measure
of $(S,g|_{S})$.
Then
$$\overline{F}(u)=\sqrt{\tfrac{\int_{y\in S}g_y(u,u)d\mu}{\int_{y\in S}d\mu}},
\quad\forall u\in \mathbf{V},$$
is the {\it averaged Euclidean norm} induced by $F$.

\begin{lemma}\label{lemma-3}
Let $F_i$ be a Minkowski norm on $\mathbf{V}_i$, $\overline{F}_i$ the
averaged Euclidean norm induced by $F_i$, and
$l:\mathbf{V}_1\rightarrow\mathbf{V}_2$ a real linear isomorphism satisfying
$F_1 = c F_2\circ l$ for some positive constant $c$. Then $\overline{F}_1 =c \overline{F}_2\circ l$.
\end{lemma}

\begin{proof}

First, we prove the lemma when $c=1$, i.e., $l$ is a linear isometry. Let $g_i$, $S_i$ and $d\mu_i$ be the Hessian metric of $F_i$, the indicatrix of $F_i$, and the volume measure of $(S_i,g_i|_{S_i})$ respectively. Then $l$ induces an isometry between $(S_i, g_i|_{S_i})$ (see Page 1488 of \cite{XM2022} or Lemma 4.1 in \cite{Xu2022}), i.e.,
$g_{2,l(y_1)}(l(u),l(u))=g_{1,y_1}(u,u)$, $\forall y_1\in S_1$, $u\in\mathbf{V}_1$. Moreover,
$l$ preserves the volume measure, i.e.,  $l^*d\mu_2=d\mu_1$. So for any $u\in \mathbf{V}$,
\begin{eqnarray*}
\overline{F}_2(l(u))^2&=&\tfrac{\int_{y_2\in S_2}g_{2,y_2}(l(u),l(u))d\mu_2}{\int_{y_2\in S_2}d\mu_2}
=\tfrac{\int_{y_1\in S_1}g_{2,l(y_1)}(l(u),l(u))l^*d\mu_2}{\int_{y_1\in S_1}l^*d\mu_2}\\
&=&\tfrac{\int_{y_1\in S_1}g_{1,y_1}(u,u)d\mu_1}{\int_{y_1\in S_1}d\mu_1}=\overline{F}_1(u)^2,
\end{eqnarray*}
i.e., $\overline{F}_1 = \overline{F}_2\circ l$.

Second, we prove the lemma when $l=\mathrm{id}$ and $F_1=c F_2$. Since $c\cdot \mathrm{id}$ is a linear isometry between $F_1$ and $F_2$, above discussion indicates that $\overline{F}_1(u)=\overline{F}_2(cu)=c\overline{F}_2(u)$.

Since any $l$ in Lemma \ref{lemma-3} is the composition of the maps in above two cases, the proof is finished.
\end{proof}

Let $F$ be a Finsler metric on $M$. We can average $F$ in each tangent space, and get a Riemannian metric $g=\overline{F}^2$ on $M$. We call this $g$  the {\it averaged Riemannian metric} induced by $F$. Using Lemma \ref{lemma-3}, we can immediately observe

\begin{lemma}\label{lemma-1}
For the averaged Riemannian metric $g$ induced by a Finsler metric $F$ on $M$, we have the following:
\begin{enumerate}
\item when $f\in C(M,F)$ satisfies $f^*F=\rho F$, then $f^*g=\rho^2 g$;
\item when $V\in \mathfrak{C}(M,F)$ satisfies $L_VF=\rho F$, then $L_V g=2\rho g$;
\item $C(M,F)\subset C(M,g)$, $C_0(M,F)\subset C_0(M,g)$, and $\mathfrak{C}(M,F)\subset\mathfrak{C}(M,g)$;
\item $I(M,F)\subset I(M,g)$, $I_0(M,F)\subset I_0(M,g)$, and $\mathfrak{I}(M,F)\subset \mathfrak{I}(M,g)$.
\end{enumerate}
\end{lemma}

This average process can help us quickly see that $I(M,F)$ is a Lie transformation group (See Section 2.3 in \cite{DX2016}).
Using a similar argument, we can prove

\begin{lemma}\label{lemma-2}
For a connected Finsler manifold $(M,F)$ with $\dim M>2$,
$C(M,F)$ is a Lie transformation group.
\end{lemma}

\begin{proof}
Let $g$ be the averaged Riemannian metric induced by $F$. By Theorem 1 in Page 310 of \cite{KN1963}, $C(M,g)$ is a Lie transformation group, i.e., $C(M,g)$ is a Lie group, and the map
$\Phi(f,x)=f(x): C(M,g)\times M\rightarrow M$
is smooth. It is easy to check that $C(M,F)$ is a transformation group acting on $M$.  Lemma \ref{lemma-1} tells us that
$C(M,F)$ is contained in $C(M,g)$. To prove Lemma \ref{lemma-2}, we only need to prove $C(M,F)$ is a closed subgroup of $C(M,g)$. Let $f_n$ be a sequence in $C(M,F)$, then we have
\begin{equation}\label{002}
\tfrac{F(x,y)}{F(x,y')}=\tfrac{F(f_n(x),(f_n)_*(y))}{F(f_n(x),(f_n)_*(y'))},\quad\forall x\in M, y,y'\in T_xM\backslash\{0\}, n\in\mathbb{N}.
\end{equation}
The smoothness
of $\Phi$ implies
$$\lim_{n\rightarrow\infty}f_n(x)=f(x),\ \lim_{n\rightarrow\infty}(f_n)_*(y)=f_*(y)\in TM\backslash0,\quad \forall x\in M, y\in T_xM\backslash 0.$$
So we can take the limit of (\ref{002}) for $n\rightarrow\infty$ and get
\begin{equation*}
\tfrac{F(x,y)}{F(x,y')}=\tfrac{F(f(x),f_*(y))}{F(f(x),f_*(y'))},\quad\forall x\in M, y,y'\in T_xM\backslash\{0\},
\end{equation*}
which guarantees $f\in C(M,F)$. This argument verifies the closeness of $C(M,F)$ in $C(M,g)$ and ends the proof.
\end{proof}

It is a slight difference between $C(M,F)$ and $I(M,F)$ that all $I(M,F)$ are Lie transformation groups \cite{DH2007}, but not all $C(M,F)$.
For those Lie transformation groups $C(M,F)$ and $I(M,F)$, their Lie algebras can be naturally identified as $\mathfrak{C}(M,F)$ and $\mathfrak{I}(M,F)$ respectively.

%
%
\subsection{Conformal fields on a compact homogeneous Finsler manifold}
\label{subsection-3-2}
Let $(M,F)$ be a compact connected homogeneous Finsler manifold. The homogeneity here means that
$I_0(M,F)$ acts transitively on $M$. When $\dim M>2$, $C_0(M,F)$ is a Lie transformation group by Lemma \ref{lemma-2}. Now we discuss $C_0(M,F)$ case by case when $\dim M\leq2$.

{\bf Case 1}: $\dim M=1$. In this case, $M=S^1$, and any Finsler metric $F$ on $M$ is
homogeneous if and only if $\lambda(x)=\sup_{y\in T_xM\backslash\{0\}}\tfrac{F(x,y)}{F(x,-y)}$
is constant for all $x\in M$. In this case, $C_0(M,F)=\mathrm{Diff}^+(S^1)$
and any smooth vector field  is conformal. This explains the necessity
for the dimension assumptions in Theorem \ref{thm-2} and Theorem \ref{cor-1}.

When $\dim M=2$, $I_0(M,F)$ is a compact connected Lie group with $\dim I_0(M,F)\leq 3$. So There are only three possibilities, a Riemannian $S^2$ with constant curvature, a Riemannian $\mathbb{R}\mathrm{P}^2$ with constant curvature,  or a locally Minkowskian $T^2$.

{\bf Case 2}:
$(M,F)$ is
the Riemannian $S^2$ with constant curvature. In this case $C_0(M,F)=PGL(2,\mathbb{C})$  can be identified as the group
of holomorphic diffeomorphisms on $\mathbb{C}\mathbf{P}^1=\mathbb{C}\coprod\{\infty\}$, i.e., each $f\in C_0(M,F)$ is represented by the map $z\mapsto \tfrac{az+b}{cz+d}$ with $ad-bc\neq0$.

{\bf Case 3}:
$(M,F)$ is a Riemannian $\mathbb{R}\mathrm{P}^2$ with constant curvature. We can identify $\mathbb{R}\mathrm{P}^2=\mathbb{C}\mathrm{P}^2/\mathbb{Z}_2$, in which $\mathbb{C}\mathrm{P}^1=\mathbb{C}\coprod\{\infty\}$ and
$\mathbb{Z}_2$
generated by the antipodal map is represented by $\theta(z)=-\overline{z}^{-1}$.
Consider any $V\in \mathfrak{C}(M,F)$. It one-to-one induces a $\theta$-invariant holomorhpic field $V'$ on $\mathbb{C}\mathrm{P}^1$, which generates a one-parameter subgroup in
$H=\{f| f\in PGL(2,\mathbb{C}), f\circ\theta=\theta\circ f\}$. This observation identifies the Lie group $C_0(M,F)$
with the identity component of $H$.
Direct calculation shows $\dim H=3$, so
$C_0(M,F)=I_0(M,F)=SO(3)$ in this case.

{\bf Case 4}: $(M,F)$ is a locally Minkowskian $T_2$.
Let $g$ be the averaged Riemannian metric induced by $F$. By Lemma \ref{lemma-2}, $(M,\overline{F})$ is a flat Riemannian torus. We may present $M$ as $M=\mathbb{C}/\Lambda$ and assume that $\mathfrak{g}$ is the standard Euclidean metric.
Using Lemma \ref{lemma-1}, any $V\in\mathfrak{C}(M,F)$ is a conformal vector field on $(M,g)$. Let $f_t(z)=\varphi(t,z)$ be the one-parameter subgroup of conformal transformations generated by $V$, in which $\varphi:\mathbb{R}\times \mathbb{R}/\Lambda\rightarrow\mathbb{R}/\Lambda$
is a smooth map, then $V$ is represented by $\tfrac{d}{dt}|_{t=0}f_t$.  Since $\varphi(t,\cdot)$ is holomorphic for each fixed $t$, i.e., $\tfrac{\partial}{\partial\overline{z}}\varphi(t,z)=0$,
we have $\tfrac{\partial}{\partial \overline{z}}(\tfrac{d}{dt}|_{t=0}f_t)=\tfrac{d}{dt}|_{t=0}\tfrac{\partial}{\partial\overline{z}}\varphi(t,z)=0$.
So $V=\tfrac{d}{dt}|_{t=0}f_t$ is
a  $\mathbb{C}$-valued holomorphic function on $\mathbb{C}/\Lambda$,
which must be constant.
That implies $\dim \mathfrak{C}(M,F)\leq2$, so $C_0(M,F)=I_0(M,F)=U(1)\times U(1)$.
Notice that when $F$ is non-Riemannian, one may
alternatively use Theorem 4.1 in \cite{MT2012} to argue that any $f\in C(M,T)$ must be homothetic. By the compactness of $M$, all homothetic transformations must be isometries.

To summarize, we have the following lemma which
explains the validity of Theorem \ref{thm-1} when $\dim M=2$.

\begin{lemma}\label{lemma-4}
Let $(M,F)$ be a compact connected homogeneous Finsler manifold with $\dim M>1$.
Then $C_0(M,F)$ is a Lie transformation group and $\mathrm{Lie}(C_0(M,F))=\mathfrak{C}(M,F)$.
\end{lemma}
\section{Proof of the main results}

Let $(M,F)$ be a compact connected homogeneous non-Riemannian Finsler manifold with $\dim M>1$. By Lemma \ref{lemma-4}, we can present $M$ as the homogeneous manifold $M=G/H=K/K\cap H$, in which $G=C_0(M,F)$ and
$K=I_0(M,F)$. Suppose ${C}_0(M,F)\neq {I}_0(M,F)$, or equivalently,   $\mathfrak{g}=\mathfrak{C}(M,F)\neq\mathfrak{k}=\mathfrak{I}(M,F)$, i.e., $(M,F)$
admits a non-Killing conformal vector field $V$.

\begin{lemma}\label{lemma-6}
Keeping all assumptions and notations in this section, then $M$ must be a homogeneous sphere, and there exists a compact connected semi simple normal subgroup $K'$ of $K$ which acts transitively on $M$.
\end{lemma}

\begin{proof}Let $g$ be the averaged Riemannian metric induced by $F$. By Lemma \ref{lemma-1},
$g$ is homogeneous Riemannian metric on $M$, and $V$ is a non-Killing conformal vector field on $(M,g)$. Theorem \ref{thm-1} tells us that
$M$ is a sphere. By the classification for homogeneous spheres \cite{Bo1940}\cite{MS1943}, $SO(n)$, $SU(n)$, $Sp(n)$ or $Spin(9)$ acts transitively on the homogeneous Finsler sphere $M=SO(n)/SO(n-1)$, $SU(n)/SU(n-1)=U(n)/U(n-1)$, $Sp(n)/Sp(n-1)=Sp(n)U(1)/Sp(n-1)U(1)=Sp(n)Sp(1)/Sp(n-1)Sp(1)$ or
$Spin(9)/Spin(7)$ respectively. Since the connected isometry groups of $S^6=G_2/SU(3)$ and $S^7=Spin(7)/G_2$ are $SO(7)$ and $SO(8)$ respectively, $G_2$ and $Spin(7)$ do not need to appear.
The proof is finished.
\end{proof}

\begin{lemma}\label{lemma-5}
Keeping all assumptions and notations in this section, then for any element $U\in \mathfrak{g}$, $\mathrm{ad}(U)=[U,\cdot]:\mathfrak{g}\rightarrow\mathfrak{g}$ is semi simple, and $B(U,U)=\mathrm{tr}(\mathrm{ad}(U)^2)\leq0$, in which $B$ is the Killing form
of $\mathfrak{g}$.
\end{lemma}

The proof of Lemma \ref{lemma-5} needs the following theorem (Finsler Confromal Lichnerowicz-Obata Conjecture).

\begin{theorem}\label{thm-3}(\cite{MRTZ2009})
Suppose that $U$ is a complete conformal vector field on a connected Finsler manifold $(M,F)$ with $\dim M>1$. Then one of the following must happen:
\begin{enumerate}
\item $U$ is a Killing vector field on $(M,\rho F)$ for some positive smooth function $\rho$ on $M$;
\item $(M,F)$ is conformally equivalent to a Riemannian sphere with constant curvature;
\item $(M,F)$ is conformally equivalent to a Minkowski space and $U$ is homothetic.
\end{enumerate}
\end{theorem}

\begin{proof}[Proof of Lemma \ref{lemma-5}]
Since $(M,F)$ is compact and non-Riemannian, only (1) in Theorem \ref{thm-3} can happen, i.e.,
we can find a positive smooth function $\rho$ on $M$, such that the conformal vector field $U\in\mathfrak{g}$ is a Killing vector field for $(M,\rho F)$. Then $U$ is contained in the Lie subalgebra $\mathfrak{g}'$ which generates $G'=I_0(M,\rho F)$.
Since $M$ is compact, $I_0(M,\rho F)$ is a compact subgroup in $G=C_0(M,F)$. So we can find an $\mathrm{Ad}(G')$-invariant inner product on $\mathfrak{g}$, for which $\mathrm{ad}(U)$ is anti self adjoint. The semi simpleness of $\mathrm{ad}(U)$ follows immediately.
Moreover, the square of $\mathrm{ad}(U)$ is semi negative definite,
so $B(U,U)=\mathrm{tr}(\mathrm{ad}(U)^2)\leq0$.
This ends the proof.
\end{proof}

Denote by $\mathfrak{r}$ the radical (i.e., the maximal solvable ideal)
of $\mathfrak{g}$.
Levi's Theorem (see Theorem 5.6.6 in \cite{HN2012}) provides a linear decomposition
$\mathfrak{g}=\mathfrak{s}+\mathfrak{r}$, in which
$\mathfrak{s}$ is a semi simple subalgebra (called a {\it Levi subalgebra}).
Notice that the Levi subalgebra $\mathfrak{s}$ of $\mathfrak{g}$ is not unique, because it can be changed by any $\mathrm{Ad}(G)$-action.

\begin{lemma}Keeping all above assumptions and notations in this section, then
$\mathfrak{s}$ is compact.
\end{lemma}

\begin{proof}The Killing form $B$ satisfies the following associative
property,
\begin{equation}\label{003}
B([W_1,W_2],W_3)+B(W_1,[W_2,W_3])=0,\quad \forall W_1,W_2,W_3\in\mathfrak{g},
\end{equation}
so it is well known that $\mathfrak{r}'=\{U| U\in\mathfrak{g}, B(U,\mathfrak{g})=0\}$ is an ideal of $\mathfrak{g}$. By Cartan's Criterion (see Theorem 5.4.20 in \cite{HN2012}), $\mathfrak{r}'$ is solvable. Let $\pi:\mathfrak{g}\rightarrow\mathfrak{g}/\mathfrak{r}'$ be
the quotient endomorphism. Then
$B$ induces a negative definite bilinear form $\overline{B}(\pi(W_1),\pi(W_2))=B(W_1,W_2)$ on $\mathfrak{g}/\mathfrak{r}'$, which satisfies the associative property (see (\ref{003})).
So $\mathfrak{g}/\mathfrak{r}'$ is compact and  $\mathfrak{g}/\mathfrak{r}'=[\mathfrak{g}/\mathfrak{r}',\mathfrak{g}/\mathfrak{r}']\oplus
\mathfrak{c}(\mathfrak{g}/\mathfrak{r}')$. Obviously, $\mathfrak{r}=\pi^{-1}(\mathfrak{c}(\mathfrak{g}/\mathfrak{r}'))$
and then $\pi$ induces a Lie algebra isomorphism from $\mathfrak{s}$ to $ [\mathfrak{g}/\mathfrak{r}',\mathfrak{g}/\mathfrak{r}']$. The compactness of $\mathfrak{s}$ is proved.
\end{proof}

\begin{proof}[Proof of Theorem \ref{thm-2}]We keep all assumptions and notations in this section, and look for a contradiction from $\mathfrak{g}=\mathfrak{C}(M,F)\neq\mathfrak{I}=\mathfrak{k}$.
Let $S$ be the Lie subgroup $\mathfrak{s}$ generates in $G$ and $K'$ the subgroup in Lemma \ref{lemma-5}. Both $S$ and $K=I_0(M,F)$ are compact and connected.
 By Theorem 14.1.3 in \cite{HN2012}, i.e., each compact subgroup of $G$ is contained in a maximal compact subgroup
and all maximal compact subgroups of $G$ are conjugate to each other,
we may assume that $S$ and $K$ are contained in
 the same maximal compact subgroup $L$ of $G$.

 {\bf Claim A}: $K'\subset S\subset K$, i.e., $S$ acts transitively and isometrically on $(M,F)$.

 Denote by $\mathfrak{l}$ the Lie algebra of $L$, which is
 a compact Lie subalgebra of $\mathfrak{g}$. Because
 $\mathfrak{l}=\mathfrak{s}+(\mathfrak{l}\cap\mathfrak{r})$, in which $\mathfrak{l}\cap\mathfrak{r}$ is a solvable ideal of $\mathfrak{l}$ and $\mathfrak{s}$ is semi simple, we have $\mathfrak{l}\cap \mathfrak{r}=\mathfrak{c}(\mathfrak{l})$ and
 $ [\mathfrak{l},\mathfrak{l}]=\mathfrak{s}$. Moreover,
 $\mathfrak{k}'=\mathrm{Lie}(K')$ is semi simple, so we have $\mathfrak{k}'=[\mathfrak{k}',\mathfrak{k}']\subset[\mathfrak{l},\mathfrak{l}]=\mathfrak{s}$,
 which implies $K'\subset S$.

Since $K'$ acts transitively on $M$, so does $S$. We can present $M$ as $M=S/S\cap H$,
where $S\cap H$ is the compact isotropy subgroup at $o\in M$.
The isotropy action of $S\cap H$ is contained in $(0,+\infty)\times O(T_oM,F(o,\cdot))$,
in which $O(T_oM,F(o,\cdot))$ is the group of all linear isometries on the Minkowski norm space $(T_oM,F(o,\cdot))$. The compactness of $S\cap H$ implies that the image of its isotropic action
is contained in $O(T_oM,F(o,\cdot))$, i.e., for any $y\in T_oM$ and $f\in S\cap H$, $F(o,y)=F(o,f_*(y))$. Now we consider any $f\in S$ and $x\in M$. The transitiveness of the $K'$-action on $M$ implies a decomposition $f=f_1\circ f_2\circ f_3$, such that
$f_1,f_3\in K'$ and $f_2\in S\cap H$ satisfy $f_3(x)=o$ and $f_1(o)=f(x)$. The $K'$-action on $M$ is isometric, so for any $y\in T_xM$,
\begin{eqnarray*}
F(x,y)&=&F(o, (f_3)_*(y))=F(o,(f_2)_*((f_3)_*(y)))\\
&=&F(f_1(o),(f_1)_*((f_2)_*((f_3)_*(y))))
=F(f(x),f_*(y)).
\end{eqnarray*}
So we have $S\subset I_0(M,F)=K$, which proves Claim A.

Because $\mathfrak{g}\neq\mathfrak{k}$, Claim A implies that $\mathfrak{r}$ is not contained in $\mathfrak{k}$. We have deduced $\mathfrak{l}\cap\mathfrak{r}=\mathfrak{c}(\mathfrak{l})$, so $[\mathfrak{s},\mathfrak{r}\cap\mathfrak{k}]\subset[\mathfrak{l},\mathfrak{l}\cap\mathfrak{r}]=0$. Further more, we claim

{\bf Claim B}: Any $U\in\mathfrak{r}$ commuting with $\mathfrak{s}$ is contained in $\mathfrak{k}$.

Let $f_t$ be the conformal transformations generated by $U$. For each $t\in\mathbb{R}$, $(f_t)_*U=U$. The assumption $[U,\mathfrak{s}]=0$ implies $f_*U=U$ for any $f\in S$.
By Claim A, $S$ acts isometrically and transitively on $(M,F)$, so $U$ is of constant length. If $U=0$, then it is obviously a Killing vector field. Otherwise $(f_t)_*U=U$ implies each conformal $f_t$
is in fact isometric. Claim B is proved.

By Claim B, we has an $\mathrm{ad}(\mathfrak{s})$-invariant
decomposition $\mathfrak{r}=\mathfrak{r}_0+\cdots+\mathfrak{r}_m$ with $m>0$, such that the
$\mathrm{ad}(\mathfrak{s})$-action is trivial on  $\mathfrak{r}_0=\mathfrak{r}\cap\mathfrak{k}$,
and nontrivial irreducible on each other $\mathfrak{r}_i$. For each $0<i\leq m$, $[\mathfrak{r}_i,\mathfrak{s}]\subset\mathfrak{r}_i$ is $\mathrm{ad}(\mathfrak{s})$-invariant,
so the irreducibility of $\mathfrak{r}_i$ provides $\mathfrak{r}_i=[\mathfrak{r}_i,\mathfrak{s}]$.
Choose $U\in\mathfrak{r}_i\backslash\{0\}$, then we have $U\notin\mathfrak{k}$.
Because $U\in[\mathfrak{r}_i,\mathfrak{s}]\subset [\mathfrak{r},\mathfrak{g}]$, $\mathrm{ad}(U):\mathfrak{g}\rightarrow\mathfrak{g}$ is nilpotent
by Corollary 5.4.15 in \cite{HN2012}. On the other hand, $\mathrm{ad}(U)$ is semi simple by Lemma \ref{lemma-5}. So $\mathrm{ad}(U)\mathfrak=0$ on $\mathfrak{g}$. In particular, we
have $[U,\mathfrak{s}]=0$, which contradicts Claim B and ends the proof.
\end{proof}

\begin{proof}[Proof of Theorem \ref{cor-1}]
When $\rho$ is a positive constant function, $\rho F$ is obviously homogeneous.
Assuming that $\rho$ is non-constant and $\rho F$ is
homogeneous. Then there exist $x,x'\in M$ with $\rho(x_1)\neq\rho(x_2)$, and $f_1\in I_0(M,F)$, $f_2\in I_0(M,\rho F)$ with $f_1(x)=f_2(x)=x'$. Obviously $f_1\notin I_0(M,\rho F)$ and $f_2\notin I_0(M,F)$. So $I_0(M,F)\subsetneq I_0(M,F)\cup I_0(M,\rho F)\subset C_0(M,F)$, which contradicts Theorem \ref{thm-2} and ends the proof.
\end{proof}

{\bf Acknowledgement}. This paper is supported by Beijing Natural Science Foundation (1222003) and National Natural Science Foundation of China (12131012,
11821101). The author sincerely thank Shaoqiang Deng, Xiaobo Liu and Mengke Wu for helpful discussions.

\end{document}